\documentclass[10pt]{article}

\usepackage{amsthm,amsmath,amsfonts,amssymb}
\usepackage[utf8]{inputenc}

\newtheorem{thm}{Th\'eor\`eme}[section]
\newtheorem{cor}[thm]{Corollaire}
\newtheorem{lem}[thm]{Lemme}
\newtheorem{prop}[thm]{Proposition}

\theoremstyle{definition}\newtheorem{defn}[thm]{D\'efinition}
\theoremstyle{definition}
\theoremstyle{definition}

\newcommand{\R}{\mathbb R}
\newcommand{\C}{\mathbb C}
\newcommand{\Z}{\mathbb Z}
\newcommand{\N}{\mathbb N}

\newcommand{\K}{\mathcal K}

\newcommand{\op}{\operatorname}

\newcommand{\mois}{%
\ifcase\month\or
 Janvier\or F\'evrier\or Mars\or Avril\or Mai\or Juin\or
 Juillet\or Ao\^ut\or Septembre\or Octobre\or Novembre\or
D\'ecembre\fi
}

\title{Classification des feuilletages moyennables par surfaces} 

\author{Miguel Berm\'udez\\
\\
\small Institut de Math\'ematiques Jussieu\\
\small Universit\'e Paris 7\\
\small (France)}

\date{}

\begin{document}

\maketitle
\begin{abstract}
On d\'emontre dans ce papier que la caract\'eristique   
  d'Euler feuillet\'ee au sens de Connes classifie les feuilletages mesur\'es ergodiques par surfaces dont l'espace de feuilles est moyennable, modulo ceux de caract\'eristique d'Euler non n\'egative. Outre le cas trivial des surfaces compactes, ces derniers sont de deux types: les actions libres ergodiques de $\mathbb C$ et les fibr\'es en cercles au dessus d'un flot r\'eel ergodique. Gr\^ace \`a des travaux pr\'ec\'edents de Rudolph, Feldman, Ornstein, Weiss et autres, on sait qu'il n'y a qu'un seul feuilletage du premier type tandis que ceux de deuxi\`eme type ne sont malheureusement pas classifiables. On g\'en\'eralise ainsi le th\'eor\`eme classique de classification des surfaces compactes. 
\end{abstract}

\section{Introduction}

Le probl\`eme de la classification des feuilletages mesur\'es est central en th\'eorie ergodique. Par exemple, la th\'eorie des feuilletages mesur\'es de dimension un est \'equivalente \`a celle des automorphismes d'un espace de Lebesgue. On passe d'une transformation au feuilletage par la m\'ethode de suspension ou {\em mapping torus} et on passe du feuilletage \`a la transformation par restriction aux transversales. L'isomorphisme de feuilletages mesur\'es de dimension un n'est autre que la bien connue \'equivalence de Kakutani entre transformations. L'un des plus importants invariants dynamiques d'une transformation mesurable, l'entropie m\'etrique, est plus pr\'ecisement un invariant du feuilletage obtenu par suspension d'apr\`es la connue formule d'Abramov \cite{Abr}. On peut prouver, comme le remarque Feldman dans \cite{Fe}, que les feuilletages de dimension un ne sont pas compl\`etement classifiables par l'entropie. Plus encore, pour tout $\alpha$ positif il existe une quantit\'e non d\'enombrable de feuilletages (non isomorphes) d'entropie $\alpha$. Les travaux d'Ornstein-Rudolph-Weiss \cite{ORW} montrent n\'eanmoins que l'entropie classifie  une importante famille de feuilletages, connus sous le nom de LB ({\em loosely Bernoulli}). On sait d\'esormais qu'il n'existe pas d'invariants alg\'ebriques capables de classifier les feuilletages de dimension un \cite{FW}.

\bigskip
En dimension deux on retrouve des invariants et des ph\'enom\`enes nouveaux. Les feuilletages mesur\'es ergodiques de dimension deux \'etant des g\'en\'eralisations naturelles de la notion de surface compacte, on pourrait na\"ivement esp\'erer les classifier gr\^ace \`a la caract\'eristique d'Euler feuillet\'ee introduite par Connes dans \cite{Con1}. Ceci n'est malheureusement pas possible, car il existe par exemple des feuilletages mesur\'es orientables de dimension deux \`a caract\'eristique d'Euler $-1$ qui sont, aussi bien du point de vue de la dynamique transverse que de la topologie des feuilles, compl\`etement diff\'erents. 

\bigskip
On rappelle d'abord que d'apr\`es des r\'esultats de Connes, Hector et moi m\^eme on sait que dans le cas des feuilletages mesur\'es ergodiques orientables par surfaces $(X,\mu)$ on a: 
\begin{enumerate}
\item $\op{Eu}(X,\mu)>0$ si et seulement si $(X,\mu)\simeq (\mathbb
  S^2, \op{Compter})$ (\cite{Con1}).
\item $\op{Eu}(X,\mu)=0$ si et seulement si $(X,\mu)$ est d\'efinie
  par une action localement libre de $\mathbb C$ (\cite{Be1},\cite{BH}). 
\end{enumerate}
o\`u $\op{Eu}$ d\'esigne la caract\'eristique d'Euler feuillet\'ee et le symbol $\simeq$  l'isomorphisme de feuilletages m\'esur\'es (voir \S2.1). En particulier, les feuilletages \`a caract\'eristique d'Euler non n\'egative sont tous moyennables. Mais il existe aussi des feuilletages moyennables \`a caract\'eristique d'Euler n\'egative, et nous avons une mani\`ere naturelle d'en construire. Soit $(X,\mu)$ un feuilletage moyennable et $T$ une transversale mesurable. En greffant une anse en chaque point de $T$ on obtient un feuilletage que nous noterons $(X,\mu)^\#_T$ et dont la caract\'eristique d'Euler est \'egale \`a $\op{Eu}(X,\mu)-2\mu(T)$. On a ainsi diminu\'e la caract\'eristique d'Euler sans changer l'espace de feuilles et donc le caract\`ere moyennable. Nous prouvons dans ce papier que tous les feuilletages moyennables de dimension deux sont en fait obtenus par cette construction:

\newtheorem*{thmP}{\sc Th\'eor\`eme A}
\begin{thmP}\label{thm:A}{\em
  Pour tout feuilletage moyennable ergodique par surfaces $(X,\mu)$ de caract\'eristique d'Euler n\'egative il existe un feuilletage ergodique   $(X_{0},\mu_{0})$ \`a caract\'eristique d'Euler nulle tel que 
  $$
  (X,\mu)\simeq (X_{0},\mu_{0})_T^\#
  $$
  avec $\mu_{0}(T)=\frac{1}{2}\op{Eu}(X,\mu)$.}
 \end{thmP}

\bigskip
Un fameux th\'eor\`eme de Ghys \cite{Gh1} qui dit que la plupart des feuilles d'un feuilletage mesur\'e ergodique ont $0$, $1$, $2$ ou un nombre infini de bouts. Ce nombre est donc un invariant du feuilletage. Pour tout nombre r\'eel $\mathbf{e}\in \R$ et tout $\mathbf{b}\in\{0,1,2,\infty\}$ on note $\Phi(\mathbf{e},\mathbf{b})$ l'espace de modules de feuilletages mesur\'es ergodiques \`a caract\'eristique d'Euler $\mathbf e$ et nombre de bouts $\mathbf b$. On peut alors r\'esumer la discussion pr\'ec\'edente dans un tableau contenant tous les feuilletages mesur\'es orientables de dimension deux:
\begin{center}
  \begin{tabular}[t]{|c||c|c|c|c|}
    \hline \, & $\mathbf e>0$ & $\mathbf e=0$ & \multicolumn{2}{c|}{$\mathbf e<0$} \\  
    \hline\hline $\mathbf b =0$ & $\Sigma^{0}$ & $\Sigma^{1}$ & \multicolumn{2}{c|}{$\Sigma^{g}=\Sigma^{1}\#\dots\# \Sigma^{1}\quad(g\geq 2)$}\\
    \hline $\mathbf b =1$ & $\varnothing$ &  $\Phi(\mathbb C)$ & $\quad \Phi(\mathbb C)^\#_T\quad $&
    $\mathbf{N.M.}$\\
    \hline $\mathbf b =2$ & $\varnothing$ & $\Phi(\R)\times \R/\Z$ &
    \multicolumn{2}{c|}{$\big(\Phi(\R)\times \R/\Z\big)^\#_T$}\\ 
    \hline $\mathbf b =\infty$ & $\varnothing$ & $\varnothing$ &
    \multicolumn{2}{c|}{$\mathbf{N.M.}$}\\ 
    \hline
  \end{tabular}
\end{center}
o\`u $\Sigma^{g}$ d\'esigne la surface compacte orientable de genre $g\in \N$ et $\Phi(\mathbb K)$ l'espace de modules de feuilletages mesur\'es ergodiques d\'efinis par une action {\bf libre} de $\mathbb K$ pour $\mathbb K=\R$ ou $\C$.  Le lecteur l'aura compris, nous avons un tableau contenant tous les feuilletages (orientables ou non) qui n'est autre qu'un ``rev\^etement \`a deux feuillets'' de celui-ci. Les feuilletages non moyennables ($\mathbf{N.M.}$) ne peuvent appartenir qu'\`a deux des neuf cases d'apr\`es les r\'esultats d'Adams \cite{Ada} et Gaboriau \cite{Gab1}. On peut facilement construire des exemples dans chaque case. Les suspensions des actions libres des groupes de surfaces de genre $\geq 2$ fournissent des feuilletages non moyennables de type $\Phi(\mathbf e,1)$ pour $\mathbf e<0$. Des feuilletages de type $\Phi(\mathbf e,\infty)$ sont obtenus par un proc\'ed\'e de grossissement des actions des groupes
discrets ayant un Cantor de bouts.

\medskip
Deux questions naturelles se posent:
\begin{enumerate}
\item Quel est l'espace de modules des feuilletages moyennables plats, i.e. \`a caract\'eristique d'Euler nulle;
\item Combien de feuilletages diff\'erents peut-on construire par chirurgie \`a partir d'un feuilletage plat donn\'e.
\end{enumerate}
Dans le cas o\`u la feuille g\'en\'erique a un seul bout, le th\'eor\`eme suivant r\'epond \`a ces deux questions.

\newtheorem*{thmB}{\sc Th\'eor\`eme B}
\begin{thmB}\label{thm:B}{\em
Tous les feuilletages mesur\'es moyennables par plans sont isomorphes. Si $(X_{0},\mu_{0})$ est un feuilletage de ce type, alors les deux conditions suivantes sont \'equivalentes:
  \begin{enumerate}
  \item $(X_{0},\mu_{0})^\#_T\simeq (X_{0},\mu_{0})^\#_S$;
  \item $\mu_{0}(T)=\mu_{0}(S)$.
  \end{enumerate}
}\end{thmB}

Pour finir, quelques mots concernant les feuilletages de type $\Phi(0,2)$, i.e. ceux dont la feuille g\'en\'erique est un cylindre. Ils sont isomorphes au produit d'un flot r\'eel et d'un cercle. Le flot r\'eel en question est unique \`a \'equivalence de Kakutani pr\`es, ce qui fait que son entropie est un invariant du feuilletage. Par cons\'equent, m\^eme si les feuilletages de ce type ne sont pas compl\'etement classifiables, l'entropie m\'etrique du flot sous-jacent fournit une partition tr\`es satisfaisante de l'ensemble de classes d'isomorphisme. La derni\`ere section de ce papier est consacr\'ee \`a l'\'etude des feuilletages \`a deux bouts.

\bigskip
Nous avons fait donc ainsi le tour de tous les feuilletages mesur\'es moyennables dont les feuilles sont des surfaces.

\section{Quelques pr\'eliminaires}

\subsection{Feuilletages mesur\'es}
Un {\em feuilletage bor\'elien} est un triplet $(X,\mathcal B,\mathcal F)$ form\'e par un ensemble $X$ muni d'une structure de Borel standard $\mathcal B$ et d'une structure de vari\'et\'e topologique  $\mathcal F$. On {\bf ne suppose pas} que la structure bor\'elienne est celle engendr\'ee par la topologie. Les composantes connexes de $\mathcal F$ sont appel\'ees {\em feuilles}. On supposera que les feuilles sont des vari\'et\'es {\bf s\'epar\'es \`a base d\'enombrable}. Les morphismes entre feuilletages bor\'eliens sont les applications qui sont simultan\'ement bor\'eliennes et continues, que nous appellons des applications BT. En particulier, les applications BT envoient feuille sur feuille. Deux feuilletages bor\'eliens sont dits isomorphes s'il existe une bijection BT entre eux dont l'inverse est BT.

\medskip
L'exemple le plus simple de feuilletage bor\'elien est le produit $V\times T$ d'une vari\'et\'e connexe $V$ et d'un espace bor\'elien standard $T$. Un tel feuilletage est appel\'e {\em prisme} de {\em base} $V$ et {\em verticale} $T$. Les feuilles de $V\times T$ sont les sous-vari\'et\'es horizontales  $V\times t$ ($t\in T$), lesquels seront appel\'es les {\em plaques} du prisme. 

\medskip
On appelle atlas de $(X,\mathcal B,\mathcal F)$ une famille d\'enombrable de bor\'eliens ferm\'es $U_i$ isomorphes \`a des prismes $\mathbb D^n\times T_i$, o\`u $\mathbb D^n$ est un disque ferm\'e de dimension $n$. On supposera d\'esormais que tous les feuilletages bor\'eliens poss\`edent un atlas. 

\medskip
On appelle transversale de $(X,\mathcal B,\mathcal F)$ un sous-ensemble $T\subset X$ qui est bor\'elien, ferm\'e et discret, autrement dit, un bor\'elien $T\in \mathcal B$ qui rencontre les feuilles de $\mathcal F$ le long de sous-espaces ferm\'es discrets. Puisque les feuilles sont suppos\'ees s\'eparables, l'intersection d'une transversale et d'une feuille est un ensemble d\'e\-nombrable. La relation d'\'equivalence {\em appartenir \`a la m\^eme feuille} induit donc une relation d'\'equivalence bor\'elienne \`a classes d\'enombrables sur $T$ au sens de \cite{FM}. Deux transversales $T$ et $S$ seront dites {\em \'equivalentes} (et on note $T\sim S$) s'il existe un isomorphisme bor\'elien $\gamma:S\to T$ qui pr\'eserve les feuilles, i.e. s'il pr\'eserve la relation d'\'equivalence induite sur $T$ par le feuilletage. Une {\em mesure  transverse} est une application $\sigma$-additive qui assigne \`a toute transversale $T$ un nombre $\mu(T)\in [0,+\infty]$. Une mesure transverse est dite {\em invariante} si elle prend la m\^eme valeur sur des transversales \'equivalentes.

\medskip
Pour simplifier, on omettra d\'esormais dans la notation, si possible, toute r\'ef\'erence explicite aux structures bor\'elienne $\mathcal B$ et topologique $\mathcal F$, et on parlera tout simplement du feuilletage $X$, des bor\'eliens de $X$ et des feuilles de $X$. On parlera de $\mathcal F$ et $\mathcal B$ comme des structures sous-jacentes si jamais on en a besoin.

\begin{defn}\label{FMF}
On appelle {\em feuilletage mesur\'ee} un couple $(X,\mu)$ form\'e par un feuilletage bor\'elien $X$ muni d'une mesure transverse invariante $\mu$. Un tel feuilletage est dit {\em fini} si la condition suivante est v\'erifi\'ee:
\begin{description}
\item[\bf (MF)] Il existe un atlas {\bf fini} de $X$ form\'e par des prismes $U_i\simeq \mathbb D^n\times T_i$ dont la verticale $T_i$ est de mesure {\bf finie} non nulle $0 < \mu(T_i)<+\infty$
\end{description} 
\end{defn}

\paragraph{Morphismes de feuilletages mesur\'es.} Sur un feuilletage mesur\'e $(X,\mu)$ nous avons une notion d'{\em partie n\'egligeable}. Une partie $B\subset X$ est dite n\'egligeable si toute transversale de $X$ contenue dans $B$ est de mesure nulle. Un sous-ensemble d'une partie n\'egligeable est bien sur n\'egligeable et il est facile \`a montrer qu'une partie de $X$ est n\'egligeable si et seulement si son satur\'e, i.e. la r\'eunion des feuilles de $X$ qui la rencontrent, est aussi n\'egligeable. Les parties n\'egligeables qui comptent sont donc les satur\'ees. On appelera {\em partie totale} de $X$ le complementaire d'une partie n\'egliegable satur\'ee. Il est raisonnable, d'un point de vue de la mesure, de ne pas distinguer un feuilletage mesur\'e de ses parties totales.

\begin{defn}\label{ISO}
Un morphisme de feuilletages mesur\'es $(X_{0},\mu_{0})$ et $(X_{1},\mu_{1})$ est donn\'e par un morphisme de feuilletages bor\'eliens entre deux parties totales $\widehat X_{0}\subset X_{0}$ et $\widehat X_{1}\subset X_{1}$. Deux feuilletages mesur\'ees  seront dits {\em isomorphes} s'il existe des parties totales qui sont isomorphes entant que feuilletages bor\'eliens. 
\end{defn}

\paragraph{Note.} On se restreint dans ce papier au cas des feuilletages de dimension deux et on suppose pour simplifier que les feuilles sont mesurablement munies d'une structure diff\'erentiable, dans le sens o\`u il existe un atlas (pas n\'ecessairement fini) avec des cartes telles que les changement de plaques sont des diff\'eomorphismes locaux. Ceci n'est pas vraiment une restriction car on peut  d\'emontrer que tous les feuilletages par surfaces poss\`edent une structure diff\'erentiable dans ce sens. Les isomorphismes de feuilletages seront suppos\'es des diff\'eomorphismes le long des feuilles, \`a moins de sp\'ecifier le contraire.

\subsection{Feuilletages par surfaces}
Soit $X$ un feuilletage bor\'elien de dimension deux.

\medskip
Une {\em triangulation} de $X$ est une famille mesurable de triangulations des feuilles au sens de \cite{BH}. Ceci signifie que l'ensemble des triangles est un espace de Borel standard $\K$ et qu'il existe une application bor\'elienne continue $\pi:\Delta^2\times\K\to X$ telle que $\pi:\Delta^2\times \{\sigma\}\to \sigma$ est un isomorphisme simplicial pour tout triangle $\sigma\in \K$. Tout feuilletage par surfaces poss\`ede une triangulation (voir \cite{BH}). On supposera d\'esormais que les feuilletages consid\'er\'es sont munis d'une triangulation $\K$.

\bigskip
Etant donn\'ee une surface $L$, on appelle {\em domaine} de $L$ toute surface compacte \`a bord contenue dans $L$. Un domaine $\Omega$ dans une feuille de $X$ sera dit {\em simplicial} s'il est r\'eunion de triangles de $\K$. 

\bigskip
Un bor\'elien $B$ de $X$ est dit {\em $\phi$-compact} si toutes ses feuilles sont des domaines simpliciaux. On dira que $B$ est une {\em pile simpliciale} s'il existe une surface triangul\'ee connexe $\Omega$ et un isomorphisme bor\'elien simplicial $\pi:\Omega\times T\to B$. Plus g\'en\'eralement un dira qu'une paire $A\subset B$ de bor\'eliens $\phi$-compacts est une {\em pile simplicial} s'il existe une surface triangul\'ee connexe compacte $\Omega$, un domaine simplicial pas forcement connexe $\Omega'\subset \Omega$ et un isomorphisme bor\'elien simplicial de paires $(\Omega,\Omega')\times T\to (B,A)$.

\bigskip
La preuve du lemme et la proposition suivants peut \^etre trouv\'ee dans \cite{BH}. 

\begin{lem}\label{lem:part-simp}
  Soit $A\subset B$ une paire de bor\'eliens $\phi$-compacts. Alors il existe une partition d\'enombrable de $B$ en bor\'eliens satur\'es $B_i$ telle que la paire $(B_i,A\cap B_i)$ est une pile simpliciale.
\end{lem}

\bigskip
Le feuilletage $X$ est dit {\em hypercompact} s'il poss\`ede une {\em filtration $\phi$-compacte}, i.e. une suite $\mathcal B=\{B_{k}|k\in \N\}$ de bor\'eliens de $X$ telle que pour chaque $k\in \N$:
\begin{enumerate}
\item $B_{k}$ est $\phi$-compact;
\item $B_{k}$ est contenu dans $B_{k+1}$;
\end{enumerate}
On dira qu'un feuilletage mesur\'ee $(X,\mathcal F,\mu)$ est {\em hypercompact} s'il admet un sous-feuilletage bor\'elien hypercompact de mesure totale. 

\bigskip
Nous avons:

\begin{prop}[\cite{BH}] 
  Un feuilletage mesur\'e est moyennable si et seulement s'il est hypercompact.
\end{prop}

\subsection{Chirurgie sur les feuilletages}
Soit $X$ un feuilletage bor\'elien \`a bord et consid\'erons $Y\subset \partial X$ une partie bor\'elienne satur\'ee du bord de $X$. Si $\phi:Y\to Y$ est un automorphisme de feuilletages bor\'eliens, on peut recoller $X$ avec lui m\^eme le long de cet isomorphisme pour obtenir un nouveau feuilletage \`a bord que nous noterons $X_{\phi}$, et dont le bord coincide avec le complementaire de $Y$ dans $\partial X$. En particulier, si $Y=\partial X$ le feuilletage obtenu n'a pas de bord.

\medskip
La construction reciproque est la suivante. Soit $Y\subset X$ un bor\'elien ferm\'e de $X$ isomorphe \`a un prisme de cercles. En d\'ecoupant $X$ le long de $Y$ on obtient un feuilletage \`a bord $X^{Y}$ dont le bord est isomorphe \`a l'union disjointe de deux copies $Y$ et d'une copie de $\partial X$. Les deux copies de $Y$ sont naturellement reli\'ees par un isomorphisme $\psi$, de sorte que le feuilletage original $X$ est obtenu \`a partir de $X^{Y}$ par la construction pr\'ec\'edente. On a donc
$$
X=X^{Y}_{\psi}
$$

\medskip
Soit $X$ un feuilletage bor\'elien sans bord et $T$ une transversale. En r\'etirant un petit disque autour de chaque point de $T$ on obtient un feuilletage bor\'elien \`a bord que l'on note $X_{T}$. Le bord de ce feuilletage est naturellement isomorphe au prisme de cercles $\R/\Z\times T$. 

\bigskip
La somme connexe de deux feuilletages $X_{0}$ et $X_{1}$ est obtenu de la fa\c{c}on suivante. On se fixe un isomorphisme bor\'elien $\varphi:T_{0}\to T_{1}$ entre deux transversales $T_{0}\subset X_{0}$ et $T_{1}\subset X_{1}$. Ceci induit un isomorphisme entre les bords de $X_{0,T_{0}}$ et $X_{1,T_{1}}$ qui nous permet de les recoller pour obtenir un feuilletage sans bord que l'on note
$$
X_{0}\;\underset{\varphi}{\#}\;X_{1}
$$
et qu'on appelle la {\em somme connexe (via $\varphi$)} des feuilletages $X_{0}$ et $X_{1}$. Le type d'isomor\-phisme des feuilletages obtenus d\'epend non seulement des transversales choisies, mais aussi de l'isormorphisme $\varphi$. Si les feuilletages d'origine avaient du bord, alors on r\'ealise une construction analogue si l'on suppose que les transversales $T_{0}$ et $T_{1}$ sont int\'erieures (i.e. ne rencontrent pas le bord). Le bord du feuilletage obtenu par somme connexe est alors la r\'eunion des deux bords.

\bigskip
Dans le cas o\`u le deuxi\`eme feuilletage est isomorphe \`a un prisme de tores $\Sigma^{2}\times T$ muni de la transversale $\{*\}\times T$, alors le feuilletage obtenu par somme connexe d\'epend de la transversale chosie mais est independant de l'isomorphisme $\varphi$. Par con\'equent il sera not\'e tout simplement $X^{\#}_{T}$. Un tel feuilletage sera dit obtenu par {\em greffe d'anses} le long de $T$.

\paragraph{En pr\'esence d'une mesure transverse.} Les constructions ci-dessus peuvent \^etre r\'ealis\'ees sur les feuilletages mesur\'es (i.e. munis d'une mesure transverse invariante), la seule condition \'etant que les isomorphismes utilis\'es pour recoller les feuilletages pr\'eservent la mesure. En particulier les deux transversales choisies dans la deuxi\`eme construction doivent avoir la m\^eme mesure. Dans le cas des feuilletages {\bf (MF)}, les deux transversales doivent \^etre de mesure finie afin d'obtenir un feuilletage {\bf (MF)}.

\section{Preuve des th\'eor\`emes}

\subsection{Th\'eor\`eme A}
La difficult\'e dans la preuve de ce th\'eor\`eme consiste \`a enlever les anses des feuilles de fa\c con mesurable. Ceci peut \^etre fait facilement dans le cas \`a $0$ bouts, i.e. quand les feuilles sont compactes. En effet, par le lemme \ref{lem:part-simp} un tel feuilletage admet un d\'ecoupage en prismes du type $\Omega\times T$, o\`u $\Omega$ est une surface \`a bord. Il est facile \`a voir que $\Omega=(\Omega_{0})_{U}^{\#}$ o\`u $\Omega_{0}$ est une surface planaire et $U$ un ensemble fini de points de $\Omega_{0}$. Par cons\'equent
$$
\Omega\times T\simeq (\Omega_{0}\times T)_{U\times T}^{\#}.
$$

\medskip
Consid\'erons maintenant le cas d'un feuilletage \`a feuilles non compactes muni d'une filtration $\phi$-compacte $\mathcal B=\{B_{k}\}_{k\in\N}$. Le bord de $\mathcal B$, i.e. le bor\'elien
$$
\partial\mathcal B :=\bigcup_{k}\partial B_{k}
$$ 
est isomorphe \`a une pile de cercles $\R/\Z\times T$. Ce cercles divisent les feuilles du feuilletage en domaines compacts; autrement dit, le feuilletage obtenu en d\'ecoupant $X$ le long de $\partial\mathcal B$, que nous notons
$$
\widehat X:=X^{\partial\mathcal B}
$$
est un feuilletage \`a feuilles compactes. On a d\'ej\`a vu qu'un tel feuilletage peut \^etre obtenu par greffe d'anses sur un feuilletage \`a feuilles planaires que l'on note $\widehat X_{0}$. Le bord de ce feuilletage co\"incidant avec celui de $\widehat X$, le processus de recollement le long du bord qui permet d'ontenir $X$ \`a partir de $\widehat X$ peut \^etre appliqu\'e \`a $\widehat X_{0}$ pour obtenir un feuilletage que nous noterons $X_{0}$.

\medskip
Malheureusement ce feuilletage n'est pas en g\'eneral planaire, car le recollement des cercles du bord entre eux peut donner lieu \`a des anses. On peut assurer que $X_{0}$ est un feuilletage planaire dans le cas o\`u la filtration $\mathcal B$ est simple. Une filtration $\phi$-compacte de $X$ est dite {\em simple} si le bord de chacune des feuilles de $\mathcal B$ a autant de composantes connexes que la feuille qui la contient a de bouts. 

\medskip
\noindent {\bf Assertion 1}: {\em Si la filtration $\phi$-compacte $\mathcal B$ est simple, alors le feuilletage $X_{0}$ construit ci-dessus est planaire. En particulier sa caract\'eristique d'Euler feuillet\'ee est nulle}.

\medskip
La demostration de cette assertion est tr\`es simple. Pour fixer les id\'ees, supposons que les feuilles de $X$ ont un bout, le cas \`a deux bouts \'etant analogue. Dans ce cas toutes les plaques de $\mathcal B$ sont des surfaces dont le bord est connexe (i.e. un cercle) et sont donc hom\'eomorphes \`a des disques auxquels on a greff\'e des anses. La filtration $\mathcal B$ induit de fa\c{c}on \'evidente une filtration $\mathcal B_{0}$ du feuilletage $X_{0}$. Les plaques de $\mathcal B_{0}$ sont obtenues en recollant le long du bord un nombre fini de feuilles de $\widehat X_{0}$, qui sont par construction des surfaces planaires. De plus, chaque cercle du bord de $\mathcal B_{0}$ est homologue \`a z\'ero dans la feuille qui le contient. Le proc\'ed\'e d\'ecrit ne produit don pas d'anses car pour ce faire il faudrait que'il existe un cercle du bord de $\mathcal B_{0}$ qui devienne non trivial dans l'homologie du recollement. 

\medskip
Le reste de la d\'emonstration passe par l'assertion suivante.

\medskip
\noindent {\bf Assertion 2}: {\em Tout feuilletage hypercompact poss\`ede une filtration $\phi$-compacte simple}.

\medskip
Comme ci-dessus, on va supposer que les feuilles de $X$ ont un bout. Le probl\`eme est de construire une filtration $\phi$-compacte dont les feuilles soient \`a bord connexe. 

\medskip
Soit $\Omega$ une surface compacte simpliciale \`a bord et soit $\Gamma\subset \Omega$ une sous-vari\'et\'e simpliciale de dimension un. On dira que $\Gamma$ {\em r\'eduit} $\Omega$ si le d\'ecoupage de $\Omega$ le long de $\Gamma$ produit une surface \`a bord connexe, i.e. un disque auquel on a greff\'e des anses. La seule posibilit\'e est que les composantes connexes de $\Gamma$ soient des segments qui rejoignent des composantes connexes distinctes du bord de $\Omega$. Si $\Gamma$ r\'eduit $\Omega$ alors en r\'etirant un petit voisinage ouvert de $\Gamma$ dans $\Omega$ on obtient une sous-surface de $\Omega$ dont le bord est connexe.

\medskip
\noindent {\bf Assertion 3}: {\em Si $L$ est une surface \`a un bout et $\Omega_{n}$ est une filtration $\phi$-compacte de $L$, alors il existe une sous-vari\'et\'e simpliciale de dimension un $\Gamma\subset L$ \`a feuilles compactes qui r\'eduit $\Omega_{n}$ pour tout $n$.}

\medskip
On d\'efinit une suite $\Gamma_{n}$ de sous-vari\'et\'es compacts de dimension un (autrement dit des r\'eunions d'un nombre fini de cercles deux \`a deux disjoints) r\'eduisant simultan\'ement les surfaces compactes $\Omega_{1}\subset \cdots\subset \Omega_{n}$ de la fa\c{c}on suivante:
\begin{enumerate}
\item $\Gamma_{0}$ est la plus petite (pour le volume simplicial) sous-vari\'et\'e r\'eduisant $\Omega_{0}$. 

\item $\widehat\Gamma_{n+1}$ est une sous-vari\'et\'e v\'erifiant les conditions suivantes:
\begin{enumerate}
\item $\widehat\Gamma_{n+1}$ r\'eduit $\Omega_{n+1}$;
\item si une feuille de $\Gamma_{n}$ rencontre $\widehat\Gamma_{n+1}$ alors elle y est enti\`erement contenue;
\item $\widehat\Gamma_{n+1}$ minimise le volume simplicial de l'intersection $\widehat\Gamma_{n+1}\cap\Gamma_{n}$ parmi celles qui v\'erifient les deux conditions pr\'ec\'edentes.
\end{enumerate}

\item Enfin on pose $\Gamma_{n+1}=\widehat\Gamma_{n+1}\cup\Gamma_{n}$.
\end{enumerate}
La sous-vari\'et\'e $\Gamma=\cup_{n}\Gamma_{n}$ v\'erifie les conditions requises par l'assertion 3. En effet, elle est simpliciale et r\'eduit tous les $\Omega_{n}$ par construction. Ses feuilles sont compactes car pour tout $n\in \N$ il existe $m>n$ tel que $\Omega_{n}$ est contenu dans l'interieur d'une sous-surface de $\Omega_{m}$ dont le bord est connexe. Il est donc possible de r\'eduire $\Omega_{m}$ sans rencontrer $\Gamma_{n}$ ni la r\'eunion des feuilles de $\Gamma_{n+1},\dots,\Gamma_{m-1}$ rencontrant celles de $\Gamma_{n}$. La condition (c) implique alors que la sous-vari\'et\'e $\widehat\Gamma_{m}$ ne rencontre pas $\Gamma_{n}$. 

\medskip
Puisque $\Gamma$ r\'eduit tous les $\Omega_{n}$, en r\'etirant un petit vosinage de $\Gamma$ dans chaque $\Omega_{n}$ on obtient une suite de surfaces $\Omega'_{n}$ dont le bord est connexe. Il est facile \`a v\'erifier que la suite $\Omega'_{n}$ forme une filtration propre de $L$.  

\medskip
Pour completer la d\'emostration du th\'eor\`eme, il suffit de realiser la construction ci-dessus feuille par feuille de fa\c{c}on mesurable, en rempla\c{c}ant la suite de surfaces $\Omega_{n}$ par la trace sur les feuilles d'une filtration $\phi$-compacte $\{B_{n}\}_{n\in \N}$. Pour obtenir une filtration $\phi$-compacte propre, il suffit de garantir que la r\'eunion des $\Gamma$ est un bor\'elien de $X$. Ceci est facile \`a garantir, car il repose sur le choix de la sous-vari\'et\'e $\widehat\Gamma_{n}$ pour chaque plaque du bor\'elien $B_{n}$. Puisque le choix porte sur les sous-vari\'et\'es simpliciales de la plaque en question, il n'y a qu'un nombre fini de choix possibles. 

\subsection{Preuve du th\'eor\`eme B}
Ce th\'eor\`eme est en fait une cons\'equence imm\'ediate des r\'esultats de Rudolph \cite{Ru}, puis de Gilbert Hector et moi m\^eme \cite{BH}. Dans ce dernier papier, on prouve que tout feuilletage par plans est d\'efini par une action ergodique essentiellement libre de $\C$ de telle sorte que l'action de $\Z\times \Z$ induite pr\'eserve une transversale choisie \`a l'avance. Dans \cite {Ru}, Rudolph prouve que deux telles actions produisent des feuilletages isomorphes. Plus pr\'ecis\'ement, Rudolph trouve un isomorphisme isotope \`a l'identit\'e. Par cons\'equent \'etant donn\'e un feuilletage moyennable par plans, et deux transversales $T$ et $S$ de m\^eme mesure, il poss\`ede un automorphisme isotope \`a l'identi\'e qui envoit $T$ sur $S$. Un tel automorphisme determine de fa\c{c}on \'evidente un isomorphisme entre les feuilletages obtenus par greffe d'anses sur $T$ et $S$.

\section{Feuilletages \`a deux bouts}
Un feuilletage mesur\'e ergodique $(X,\mu)$ dont la feuille g\'en\'erique a deux bouts est automatiquement moyennable. On peut dire encore plus: d'apr\`es le th\'eor\`eme C de \cite{Gh1}, il se projette sur un feuilletage mesur\'e de dimension un $(Y,\nu)$ par une application bor\'elienne continue $\pi:X\to Y$ \`a fibres compactes qui envoit les deux bouts de $X$ sur les deux bouts de $Y$ et telle que $\pi_{*}\mu=\nu$. Nous dirons que $Y$ est un {\em ecrassement} de $X$.

\medskip
\begin{prop}
Tout \'ecrassement $\pi: (X,\mu)\to (Y,\nu)$ d'un feuilletage \`a deux bouts poss\`ede une section.
\end{prop}
\begin{proof}
La construction d'une telle section repose sur des techniques developp\'ees en d\'etail dans \cite{Be1}. L'id\'ee de la preuve est la suivante. Etant donn\'ee une transversale de mesure positive $S$ de $Y$, on peut facilement construire une section de $\pi$ sur $S$. Mais par ergodicit\'e on sait que $S$ d\'ecoupe $Y$ en des segments compacts et il n'y a pas d'obstruction pour prolonger contin\^ument la section sur chaque segment. On montre dans \cite{Be1} que cette extension peut \^etre faite de mani\`ere mesurable.
\end{proof}

Les m\^emes thecniques permettent de d\'emontrer qu'un feuilletage mesur\'e ergodique de dimension un $(Y,\nu)$ poss\`ede toujours une orientation. Le choix d'une telle orientation induit sur toute transversale de mesure positive $T$ une transformation $\gamma_{T}:T\to T$ qui envoit un point $x\in T$ sur le premier point de $\gamma_{T}(x)\in T$ obtenu en se promenant sur la feuille qui contient $x$ dans la direction dict\'ee par l'orientation. On appelle $\gamma_{T}$ l'application de premier retour sur $T$

\medskip
Consid\'erons un \'ecrassement $\pi:(X,\mu)\to (Y,\nu)$. Etant donn\'ee une transversale $S$ de $Y$, on peut supposer, quitte \`a d\'eformer leg\`erement l'application $\pi$, que $S$ est un ensemble de valeurs r\'eguliers de $\pi$. Dans ce cas $\pi^{-1}(S)$ est un feuilletage dont les feuilles sont des hypersurfaces compactes de $X$. Si $X$ est de dimension un, $T=\pi^{-1}(S)$ est en fait une transversale de $X$. De plus une orientation de $Y$ ordonne les deux bouts des feuilles de $Y$, et donc des feuilles de $X$. On consid\`ere 
$$
\gamma:T\to T
$$
l'application de premier retour sur $T$, puis on note $x'=\gamma(x)$. Par continuit\'e de $\pi$ on a pour tout $x\in T$  
$$
\pi(x)\leq \pi(x')
$$
our l'ordre induit sur les feuilles de $Y$ par leur orientation. Quitte \`a d\'eformer leg\`erement $\pi$ on peut supposer qu'on a $\pi(x)<\pi(x')$ pour tout $x\in T$. On construit alors une application bor\'elienne continue
$$
\pi':X\to Y
$$
qui envoit hom\'eomorphiquement le segment $[x,x']$ sur le segment $[\pi(x),\pi(x')]$. L'application ainsi construite est clairement homotope \`a $\pi$ et bijective puisque continue strictement croissante et surjective. Nous avons ainsi d\'emontr\'e le r\'esultat suivant :

\begin{prop}
Soit $(Y,\nu)$ et $(Y',\nu')$ deux feuilletages ergodiques de dimension un et soit $\pi:(Y,\nu)\to (Y',\nu')$ un \'ecrassement de $(Y,\nu)$. Alors $\pi$ est homotope \`a un isomorphisme de feuilletages mesur\'es.
\end{prop}

\medskip
Les deux propositions pr\'ec\'edentes permettent de d\'emontrer le th\'eor\`eme suivant:

\newtheorem*{thmC}{Th\'eor\`eme C}
\begin{thmC}{\em 
Soit $(X,\mu)$ un feuilletage mesur\'e ergodique dont la feuille g\'en\'erique a deux bouts et soient
$$
\pi_{0}:(X,\mu)\to (Y_{0},\nu_{0})\quad,\quad \pi_{2}:(X,\mu)\to (Y_{1},\nu_{1})
$$
deux \'ecrassements de $(X,\mu)$. Alors les feuilletages $(Y_{0},\nu_{0})$ et $(Y_{1},\nu_{1})$ sont isomorphes.
}\end{thmC}

En effet, \'etant donn\'ee une sectio $s_{0}:Y_{0}\to X$ de $\pi_{0}$, l'application 
$$
\pi_{1}\circ s_{0}:Y_{0}\to Y_{1}
$$
est un \'ecrassement de $Y_{0}$, donc homotope \`a un isomorphisme. En particulier, le type d'isomorphisme d'un feuilletage mesur\'e \`a deux bouts est un invariant du feuilletage.

\medskip
Dans le cas des feuilletages par cylindres, on obtient un r\'esultat plus fort:
\begin{cor}
Soit $(X,\mu)$ un feuilletage mesur\'e ergodique par cylindres et $(Y,\nu)$ son \'ecrassement. Alors on a
$$
(X,\mu)\simeq (Y,\nu)\times \mathbb S^{1}
$$
o\`u $\mathbb S^{1}$ est un cercle.
\end{cor}

\subsection{Entropie d'un feuilletage \`a deux bouts.} 
On fini en donnant la d\'efinition d'entropie d'un feuilletage \`a deux bouts, qui raffine la classification donn\'ee plus haut en fonction de la caract\'eristique d'Euler feuillet\'ee.

\medskip
Soit $(Y,\nu)$ un feuilletage ergodique de dimension un muni d'une orientation. On d\'efinit l'entropie d'une transversale $T$ comme \'etant
$$
h(T)=h(\gamma_{T})\mu(T)
$$
o\`u $h(\gamma)$ est l'entropie (au sens de Kolmogorov-Sinai \cite{Kol}) de l'application de premier retour $\gamma_{T}$ agissant sur l'espace bor\'elien $T$ muni de la mesure de probabilit\'e $\hat\mu=\mu(T)^{-1}\nu$. Puisque $(Y,\nu)$ est suppos\'e ergodique, les transformations $\gamma_{T}$, o\`u $T$ parcourt les transversales de mesure positive de $(Y,\nu)$, sont \'equivalentes au sens de Kakutani. En effet, \'etant donn\'ees deux transversales de mesure non nulle $T$ et $S$, il existe des transversales disjointes de mesure non nulle $T'\subset T$ et $S'\subset S$ telles que l'application de premier retour de $S'\cup T'$  envoit $S'$ sur $T'$ et $T'$ sur $S'$. En particulier on a 
$$
\gamma_{S'\cup T'}^{2}=\gamma_{S'}\cup\gamma_{T'}
$$
ce qui implique que les applications de premier retour $\gamma_{S'}$ et sur $\gamma_{T'}$ sont conjugu\'ees. En particulier $h(\gamma_{S'})=h(\gamma_{T'})$. Mais d'apr\`es la formule d'Abramov \cite{Abr} sur l'entropie des transformations induites on a 
$$
h(\gamma_{T})\mu(T)=h(\gamma_{T'})\mu(T')=h(\gamma_{S'})\mu(S')=h(\gamma_{S})\mu(S).
$$
Ceci montre que l'entropie d'une transversale est ind\'ependante de la transversale chosie. Il s'agit donc d'un invariant du feuilletage orient\'e.

\begin{defn}
On d\'efinit l'{\em entropie} d'un feuilletage ergodique orient\'e de dimension un par
$$
h(Y,\nu)=h(\gamma_{T})\mu(T)
$$
o\`u $T$ est une transversale quelconque de mesure non nulle.
\end{defn}

D'apr\`es le th\'eor\`eme C on peut d\'efinir l'entropie d'un feuilletage mesur\'e ergodique \`a deux bouts comme \'etant celle de son \'ecrassement. Ce nombre d\'epend bien entendu de l'orientation choisie, mais fournit une classification fine des feuilletages ergodiques \`a deux bouts.

%
%

\end{document}